\documentclass[12pt,reqno]{amsart}

\usepackage{color} 
\usepackage{amstext} \usepackage{amsthm} \usepackage{amsmath} \usepackage{amssymb}
\usepackage{latexsym} \usepackage{amsfonts} \usepackage{graphicx} \usepackage{texdraw} \usepackage{graphpap}
\usepackage{enumerate}

\usepackage[pagebackref,hypertexnames=false, colorlinks, citecolor=red, linkcolor=red]{hyperref}
\usepackage[backrefs]{amsrefs}

\usepackage[inline,nomargin]{fixme}

\input txdtools

\begin{document} 
\newcommand{\ci}[1]{_{ {}_{\scriptstyle #1}}}

\newcommand{\norm}[1]{\ensuremath{\left\|#1\right\|}} \newcommand{\abs}[1]{\ensuremath{\left\vert#1\right\vert}}
\newcommand{\ip}[2]{\ensuremath{\left\langle#1,#2\right\rangle}} \newcommand{\p}{\ensuremath{\partial}}
\newcommand{\pr}{\mathcal{P}}

\newcommand{\pbar}{\ensuremath{\bar{\partial}}} \newcommand{\db}{\overline\partial} \newcommand{\D}{\mathbb{D}}
\newcommand{\B}{\mathbb{B}}   \newcommand{\Sn}{{\mathbb{S}_n}} \newcommand{\T}{\mathbb{T}} \newcommand{\R}{\mathbb{R}}
\newcommand{\Z}{\mathbb{Z}} \newcommand{\C}{\mathbb{C}} \newcommand{\N}{\mathbb{N}} 

\newcommand{\td}{\widetilde\Delta}

\newcommand{\Aa}{\mathcal{A}} \newcommand{\BB}{\mathcal{B}} \newcommand{\HH}{\mathcal{H}} \newcommand{\KK}{\mathcal{K}} \newcommand{\DD}{\mathcal{D}}
\newcommand{\LL}{\mathcal{L}} \newcommand{\MM}{\mathcal{M}}  \newcommand{\FF}{\mathcal{F}}  \newcommand{\GG}{\mathcal{G}}  \newcommand{\TT}{\mathcal{T}}
 \newcommand{\UU}{\mathcal{U}}

\newcommand{\AapBerg}{\Aa_p(\B_n)} \newcommand{\AaFock}{\Aa_\phi (\C^n)} \newcommand{\AatwoBerg}{\Aa_2(\B_n)}
\newcommand{\Om}{\Omega} \newcommand{\La}{\Lambda} \newcommand{\AaordFock}{\Aa (\C^n)}

\newcommand{\rk}{\operatorname{rk}} \newcommand{\card}{\operatorname{card}} \newcommand{\ran}{\operatorname{Ran}}
\newcommand{\osc}{\operatorname{OSC}} \newcommand{\im}{\operatorname{Im}} \newcommand{\re}{\operatorname{Re}}
\newcommand{\tr}{\operatorname{tr}} \newcommand{\vf}{\varphi} \newcommand{\f}[2]{\ensuremath{\frac{#1}{#2}}}

\newcommand{\kzp}{k_z^{(p,\alpha)}} \newcommand{\klp}{k_{\lambda_i}^{(p,\alpha)}} \newcommand{\TTp}{\mathcal{T}_p}

\newcommand{\vp}{\varphi} \newcommand{\al}{\alpha} \newcommand{\be}{\beta} \newcommand{\la}{\lambda}
\newcommand{\tf}{\tilde{f}}
\newcommand{\li}{\lambda_i} \newcommand{\lb}{\lambda_{\beta}} \newcommand{\Bo}{\mathcal{B}(\Omega)}
\newcommand{\Bbp}{\mathcal{B}_{\beta}^{p}} \newcommand{\Bbt}{\mathcal{B}(\Omega)} \newcommand{\Lbt}{L_{\beta}^{2}}
\newcommand{\Kz}{K_z} \newcommand{\kz}{k_z} \newcommand{\Kl}{K_{\lambda_i}} \newcommand{\kl}{k_{\lambda_i}}
\newcommand{\Kw}{K_w} \newcommand{\kw}{k_w} \newcommand{\Kbz}{K_z} \newcommand{\Kbl}{K_{\lambda_i}}
\newcommand{\kbz}{k_z} \newcommand{\kbl}{k_{\lambda_i}} \newcommand{\Kbw}{K_w} \newcommand{\kbw}{k_w}
\newcommand{\BL}{\mathcal{L}\left(\mathcal{B}(\Omega), L^2(\Om;d\sigma)\right)}
\newcommand{\Fpphi}{\ensuremath{{\mathcal{F}}_\phi ^p }}
\newcommand{\Ftwophi}{\ensuremath{{\mathcal{F}}_\phi ^2 }}
\newcommand{\incn}{\ensuremath{\int_{\C}}}
\newcommand{\Finfphi}{\ensuremath{\mathcal{F}_\phi ^\infty }}
\newcommand{\Fp}{\ensuremath{\mathcal{F} ^p }} \newcommand{\Fq}{\ensuremath{\mathcal{F} ^q }}
\newcommand{\Ft}{\ensuremath{\mathcal{F} ^2 }} \newcommand{\Lt}{\ensuremath{L ^2 }}
\newcommand{\Lp}{\ensuremath{L ^p }}
\newcommand{\Fonephi}{\ensuremath{\mathcal{F}_\phi ^1 }}
\newcommand{\Lpphi}{\ensuremath{L_\phi ^p}}
\newcommand{\Ltwophi}{\ensuremath{L_\phi ^2}}
\newcommand{\Lonephi}{\ensuremath{L_\phi ^1}}
\newcommand{\af}{\mathfrak{a}} \newcommand{\bb}{\mathfrak{b}} \newcommand{\cc}{\mathfrak{c}}
\newcommand{\Fqphi}{\ensuremath{\mathcal{F}_\phi ^q }}

\newcommand{\entrylabel}[1]{\mbox{#1}\hfill}

\newenvironment{entry} {\begin{list}{X}%
  {\renewcommand{\makelabel}{\entrylabel}%
      \setlength{\labelwidth}{55pt}%
      \setlength{\leftmargin}{\labelwidth}
      \addtolength{\leftmargin}{\labelsep}%
   }%
}


\numberwithin{equation}{section}

\newtheorem{thm}{Theorem}[section] \newtheorem{lm}[thm]{Lemma} \newtheorem{cor}[thm]{Corollary}
\newtheorem{conj}[thm]{Conjecture} \newtheorem{prob}[thm]{Problem} \newtheorem{prop}[thm]{Proposition}
\newtheorem*{prop*}{Proposition}

\theoremstyle{remark} \newtheorem{rem}[thm]{Remark} \newtheorem*{rem*}{Remark} \newtheorem{example}[thm]{Example}

\theoremstyle{definition} \newtheorem{definition}[thm]{Definition}

\title{Localized frames and Compactness}

\author[F. Batayneh]{Fawwaz Batayneh} \address{Fawwaz Batayneh, Department of Mathematical Sciences \\ Clemson University \\ O-110 Martin Hall, Box 340975\\ Clemson, SC USA 29634} \email{fbatayn@g.clemson.edu}

\author[M. Mitkovski]{Mishko Mitkovski$^\dagger$} \address{Mishko Mitkovski, Department of Mathematical Sciences\\
Clemson University\\ O-110 Martin Hall, Box 340975\\ Clemson, SC USA 29634} \email{mmitkov@clemson.edu}
\urladdr{http://people.clemson.edu/~mmitkov/} \thanks{$\dagger$ Research supported in part by National Science Foundation
DMS grant \# 1101251.}

\subjclass[2000]{42B35, 43A22, 47B35, 47B38} \keywords{Localized frames, Localized operator, Multipliers,
Toeplitz operator}

\begin{abstract} We introduce the concept of weak localization for continuous frames and use this concept to define a class of weakly localized operators. This class contains many important classes of operators, including: Short Time Fourier Transform multipliers, Calderon-Toeplitz operators, Toeplitz operators on various functions spaces, Anti-Wick operators, some pseudodifferential operators, some Calderon-Zygmund operators, and many others. In this paper, we study the boundedness and compactness of weakly localized operators. In particular, we provide a characterization of  compactness for weakly localized operators in terms of the behavior of their Berezin transforms. \end{abstract}

\maketitle

\section{Introduction and Preliminaries}

The main goal of this paper is to develop a general setting in which boundedness and compactness of a large class of operators can be determined by their behavior on a very restricted special class of elements in the underlying Hilbert space. The usual setting for this type of questions are the classical function spaces of Bergman type. It is well known that in these spaces the boundedness and compactness of a wide variety of Hankel and Toeplitz operators can be determined solely in terms of their behavior on the reproducing kernels~\cite{AZ, BI, I, IMW, E, MW, MSW, Sua, XZ}. In this paper, we show that results of this kind are not special to Bergman-type spaces nor to Toeplitz and Hankel operators but hold in a much greater generality. 

To explain our results we start with the following almost trivial example. Let $\HH$ be a separable Hilbert space and let $\{e_n\}$ be a fixed orthonormal basis in $\HH$. Each bounded operator $T:\HH\to \HH$ clearly satisfies $\sup_n\norm{Te_n}<\infty$ and, moreover, if $T$ is compact then $\norm{Te_n}\to 0$ as $n\to\infty$. Even though the converse obviously fails in general, there are still  some operators for which the converse does hold; for example operators of the form $Tf=\sum_n a_n\ip{f}{e_n}e_n$. The goal of this paper is to show that a similar type of results can be obtained by replacing the orthonormal basis $\{e_n\}$ with a very general class of continuous frames. In this general setting, we will provide a wide class of operators whose boundedness and compactness can be determined by only testing on the elements of the continuous frame.


\subsection{Continuous frames} Let $\HH$ be a separable Hilbert space and let $(X, d, \la)$ be a locally compact metric measure space with a Radon measure $\la$.  
\begin{definition}\label{frame} A family $\{f_x\}_{x\in X}$ of vectors in a Hilbert space $\HH$ is said to be a \emph{continuous frame} if the following two properties hold
\begin{itemize}
\item[(i)] the mapping $x\mapsto f_x$ is bounded and continuous (and hence measurable), 
\item[(ii)] there exist constants $c, C>0$ such that for every $f\in \HH$ the following inequalities hold
$$c\norm{f}^2\leq \int_X \abs{\ip{f}{f_x}}^2d\la(x)\leq C\norm{f}^2.$$
\end{itemize}
\end{definition}

\begin{rem} Most often only weak measurability is required in (i) and $(X, d, \la)$ is only assumed to be a topological space with a Radon measure $\la$. We include the metric $d$ in the definition above just to make the concept of localization more transparent. These slightly stronger assumptions don't exclude any of the important examples of continuous frames. 
\end{rem}

 Notice  that we don't assume any type of continuity on our metric space. The name \emph{continuous frames} is used just to stress the analogy with the usual (discrete) frames. Namely, for $X=\Z$ (with the usual metric and the counting measure) this definition reduces to the usual definition of frames.  Even though the concept of continuous frames has been around for quite some time now, see, e.g.,~\cite{AAG}, so far there is no established standard terminology and other names for the same notion can be found in the literature, e.g., ``continuous resolution of the identity'', ``generalized coherent states'', etc. 

As in the classical case, the continuous frame is said to be Parseval if $c=C=1$. If only the right side of the frame inequality holds then we will say that  $\{f_x\}_{x\in X}$ forms a Bessel family. For a given continuous frame $\{f_x\}_{x\in X}$, as usual, one defines the analysis operator $T: \HH\to L^2(X,d\la)$ by $Tf:=\ip{f}{f_x}$ and (its adjoint) the synthesis operator $T^*: L^2(X,d\la)\to \HH$ by $T^*(a):=\int_X a(x)f_xd\la(x)$. Here and elsewhere the integral of a Hilbert space-valued function will be defined in the weak sense. For example, $T^*(a)$ is the unique element in $\HH$ such that $$\ip{T^*(a)}{g}=\int_X a(x)\ip{f_x}{g}d\la(x),$$ for all $g\in\HH$. The existence and uniqueness of this element is guaranteed by the Riesz representation theorem. Next, using the analysis and the synthesis operators one defines the frame operator $S: \HH\to\HH$ by $S=(T^*T)^{1/2}$ and the canonical dual continuous frame $\tf_x=S^{-1}f_x$. It is easy to see that Parseval continuous frames  coincide with their duals, i.e., $f_x=\tilde{f}_x$.

\subsection{Requirements on the indexing metric measure space} We will need to impose several requirements on the indexing metric measure space $(X,d,\la)$. We first concentrate on the metric space $(X,d)$. 

\begin{itemize}
\item[(M1)] We assume that $(X,d)$ is a locally compact, complete, and a geodesic metric space. 
\end{itemize}

As a consequence, $(X,d)$ is proper, i.e., each closed bounded set in $X$ is compact. Another consequence of these assumptions is that the closure of each open ball in $X$ is the corresponding closed ball.

The possibility to cover the underlying metric measure space $(X,d,\la)$ in a certain specific way plays a crucial role in almost all of our results. The covering property that we will require our metric measure space to possess is that of a \emph{finite asymptotic dimension}.  

\begin{definition}\label{Covering}
We will say that a metric space $(X,d)$ has a \emph{finite asymptotic dimension} if there exists an integer $N>0$ such that for any $r>0$ there is a covering $\DD_r=\{F_j\}$ of $X$ by disjoint Borel sets satisfying:
\begin{enumerate}
\item[\label{Finite} \textnormal{(i)}] every point of $X$ belongs to at most $N$ of the sets $G_j:=\{x\in X: d(x, F_j)\leq r\}$,
\item[\label{Diameter} \textnormal{(ii)}] there exists $K>0$ such that $\textnormal{diam} \, F_j\leq Kr$ for every $j$.
\end{enumerate}
\end{definition}

\begin{rem} Our requirements are slightly stronger than the ones given in the original definition by Gromov~\cite{Gr}, but for separable metric spaces the concepts are essentially the same. 
\end{rem}

\begin{itemize}
\item[(M2)] We assume that $(X,d)$ has a finite asymptotic dimension in the sense of Definition~\ref{Covering}. 
\end{itemize}

It is not hard to prove that each doubling metric measure space has a finite asymptotic dimension (see e.g.~\cite{MW}). There are however many natural metric measure spaces which are not doubling, one important example being the unit ball in $\C^n$ equipped with the Bergman metric and the hyperbolic measure. However, the following result of Gromov (proved explicitly by Roe~\cite{Roe}) shows that the unit ball as well as many other non-doubling spaces also have finite asymptotic dimension. Before stating the result we need to introduce some terminology. 

Let $o$ be some fixed point in $X$ (basepoint). For $x,y\in X$, their \emph{Gromov product} $(x|y)$ is defined by $(x|y)=(d(x,o)+d(y,o)-d(x,y))/2$. The metric space $(X,d)$ is said to be \emph{Gromov hyperbolic} if there exists $\delta>0$ such that $(x|z)\geq \min\{(x|y), (y|z)\}-\delta$, for all $x,y,z\in X$. Gromov hyperbolic spaces form a large and well studied class of metric spaces. They include all complete simply connected Riemannian manifolds whose sectional curvature is everywhere less than a fixed negative constant. A metric space $(X,d)$ is said to be of \emph{bounded growth} if for each $r>0$ there exists $M_r$ such that for every $R>0$ each ball of radius $R+r$ in $X$ can be covered by at most $M_r$ balls of radius $R$.

\begin{thm}\cite{Gr, Roe}\label{GrRoe} If $(X,d)$ is a Gromov hyperbolic geodesic metric space with bounded growth, then $X$ has a finite asymptotic dimension. 
\end{thm}  

Finally, the only requirement on the measure $\la$ (besides being Radon) that we impose is the following one. 

\begin{itemize}
\item[(M3)] We assume that $D_r:=\sup_{x\in X}\la(D(x,r))<\infty,$ for each $r>0$.
\end{itemize}

\subsection{Weakly localized frames} 

Let $\FF=\{f_x\}_{x\in X}$ and  $\GG=\{g_x\}_{x\in X}$ be two families of vectors in $\HH$ indexed by the metric measure space $(X, d, \la)$. There are various concepts of frame localization that appear in the literature. One of the most useful ones is that of $L^1$-localization. It can be defined when the indexing metric measure space possesses a group structure relative to which both the metric and the measure are invariant. In this case the pair $(\FF, \GG)$ is said to be $L^1$-localized if there exists a function $a\in L^1(X, d\la)$ such that $\abs{\ip{f_x}{g_y}}\leq a(y^{-1}x)$ for all $x, y\in X$. This concept was essentially introduced by Gr\"ochenig~\cite{Gro1} for a special (but important) class of integrable functions $a$. In the general form above the concept of $L^1$-localization was first introduced and studied by Fornasier and Rauhut~\cite{FR}. Localization of frames turned out to be crucial in many later works, including that of Balan et al~\cite{BCHL1, BCHL2}, Futamura~\cite{Fut}, Gr\"ochenig and Piotrowski~\cite{GP}, Gr\"ochenig and Rzeszotnik~\cite{GrR}, and many others. In essence, the concept of localization of frames is closely related to almost diagonality studied in the context of atomic and molecular decompositions of singular integral operators by Coifman and Mayer~\cite{CM}, Frazier, Jawerth, and Weiss~\cite{FJW}, Grafakos and Torres~\cite{GT, Tor}, Yuan, Sickel, and Yang~\cite{YSY}, and many others. 

In this paper we introduce a slightly weaker notion of localization which allows us to treat one more important function space, the Bergman space. Namely, the continuous frame of normalized reproducing kernels in the Bergman space is not $L^1$-localized, but it is localized in our weaker sense (see more about this example in Section~\ref{examples}). A crucial difference is the fact that we don't require a pointwise bound on $\abs{\ip{f_x}{g_y}}$ while all the previous localization concepts seemingly do. Our weaker notion of localization has a clear advantage in that it allows treatment of larger classes of continuous frames and localized operators (see below). However, on the negative side, some nice properties that the algebras of $L^1$-localized operators possess, such as inverse-closedness, are no longer true for our wider class of localized operators (see~\cite{Tes}). 

\begin{definition}
Given a positive measurable function $p: X\to (0,\infty)$ we say that the pair $(\FF, \GG)$ is \emph{$p$-weakly localized} if

\begin{equation}\label{Sch1}\int_X\abs{\ip{f_x}{g_y}}p(y)d\la(y)\lesssim p(x), \hspace{.5cm} \int_X\abs{\ip{f_x}{g_y}}p(x)d\la(x)\lesssim p(y), \end{equation} and 

\begin{equation}\label{Sch2}\lim_{R\to\infty}\sup_{x\in X}\frac{1}{p(x)}\int_{D(x,R)^c}\abs{\ip{f_x}{g_y}}p(y)d\la(y)=0, \hspace{0.3cm} \lim_{R\to\infty}\sup_{y\in X}\frac{1}{p(y)}\int_{D(y,R)^c}\abs{\ip{f_x}{g_y}}p(x)d\la(x)=0.\end{equation}
\end{definition}
Notice that this definition is symmetric in the sense that $(\FF, \GG)$ is $p$-weakly localized if and only if $(\GG, \FF)$ is $p$-weakly localized. We will say that the continuous frame $\{f_x\}_{x\in X}$ is $p$-weakly localized if the pair $(\tilde{\FF}, \FF)$ is $p$-weakly localized, where  $\tilde{\FF}=\{\tf_x\}_{x\in X}$ is the canonical dual frame of $\FF$.

The term \emph{localized} is related to the equalities~\eqref{Sch2}. To measure this localization more precisely we use the following function

\begin{equation}\label{rho}
\rho(\epsilon):=\inf \left\{R>0: \int_{D(x,R)^c}\abs{\ip{f_x}{g_y}}p(y)d\la(y)\leq \epsilon p(x),  \int_{D(y,R)^c}\abs{\ip{f_x}{g_y}}p(x)d\la(x)\leq \epsilon p(y)\right\}.
\end{equation}
Clearly, if the pair $(\FF, \GG)$ is $p$-weakly localized, then $\rho$ is a decreasing function such that $\rho(\epsilon)<\infty$ for all $\epsilon>0$. We are mostly interested in the behavior of this function near $0$, i.e., when $\epsilon>0$ is small. For well localized frames $\rho(\epsilon)$ decays faster as $\epsilon$ grows away from $0$. We will use this function to estimate the norm and the essential norm of weakly localized operators.

\subsection{Weakly localized operators} 

We now introduce the class of operators that will be the most important object of our study.  

\begin{definition}
We will say that the linear operator $T:\HH\to \HH$ is \emph{$p$-weakly localized with respect to the pair} $(\FF, \GG)$ if the pair $(\TT, \GG)$  is $p$-weakly localized, where $\TT=\{Tf_x\}_{x\in X}$. 
\end{definition}

Notice that $p$-weak localization with respect to  $(\FF, \GG)$ is in general different from $p$-weak localization with respect to  $(\GG, \FF)$. In Section~\ref{wlo} we will prove that the class of $p$-weak localized operators forms an algebra which can be viewed as an analog of the result of Futamura~\cite{Fut} in our more general situation.  

A very important subclass of the class of $p$-weakly localized operators is the one of so called multipliers~\cite{Bal, Bal1}. They are defined in the following way. Let $\FF=\{f_x\}_{x\in X}$ and $\GG=\{g_x\}_{x\in X}$ be two Bessel families in $\HH$ such that the pair $(\FF, \GG)$ is $p$-weakly localized. For every $u\in L^\infty(X)$ let $T_u: \HH\to\HH$ be the linear operator defined by
$$T_uf=\int_X u(x)\ip{f}{g_x}f_xd\la(x).$$ For obvious reasons, we will call all such operators multiplier operators or just multipliers. This class includes the Short Time Fourier Transform (STFT) multipliers, the class of Toeplitz operators on various functions spaces, and many others (more details are given in Section~\ref{examples}).
The results in this paper provide norm and essential norm estimates for general $p$-weakly localized operators. In particular, we give a characterization of compactness for $p$-weakly localized operators solely in terms of their behavior on the continuous frame $\FF$ and/or its dual. The techniques that we use are in essence similar to the ones used in~\cite{IMW, MW}. However, at this level of generality our results show that the results that were seemingly very much special to the class of Toeplitz operators acting on Bergman-type spaces can actually be extended to a much bigger class of seemingly unrelated operators. 

The paper is organized as follows. In the next section we give some simple preliminary results regarding the class of $p$-weakly localized operators. The main results are proved in Sections~\ref{bwlo},~\ref{cwlo}, and~\ref{Berezin}. In the last section we give some concrete examples where our results can be applied. 


\section{Basic Properties of Weakly Localized Operators} \label{wlo}

The following lemma will be used to show few basic properties of weakly localized operators.

\begin{lm}\label{useful} Let $\FF^i=\{f^i_x\}_{x\in X}, \GG^i=\{g^i_x\}_{x\in X}$, $i=1,2$ be four families in $\HH$ such that the pairs $(\FF^1, \GG^1)$ and $(\FF^2, \GG^2)$ are $p$-weakly localized. If $$\abs{\ip{k_y}{l_z}}\lesssim \int_X\abs{\ip{f^1_y}{g^1_x}}\abs{\ip{g^2_x}{f^2_z}}d\la(x),$$ for all $y, z\in X$, then the pair $(\KK, \LL)$ is $p$-weakly localized, 
where $\KK=\{k_x\}_{x\in X}, \LL=\{l_x\}_{x\in X}$. 
\end{lm}

\begin{proof} Since the pairs $(\FF^i, \GG^i)$ are both $p$-weakly localized we easily obtain $$\int_X\abs{\ip{k_y}{l_z}}p(y)d\la(y)\lesssim \int_X\int_X\abs{\ip{f^1_y}{g^1_x}}p(y)d\la(y)\abs{\ip{g^2_x}{f^2_z}}d\la(x)\lesssim p(z).$$ Similarly, $\int_X\abs{\ip{k_y}{l_z}}p(z)d\la(z)\lesssim p(y)$.

It remains to show that the equalities \eqref{Sch2} hold for the pair $(\KK,\LL)$. We show one of them, the other one being similar. Denote $$I(x,y):=\abs{\ip{f^1_y}{g^1_x}}\int_{D(y, R)^c}\abs{\ip{g^2_x}{f^2_z}}p(z)d\la(z).$$ Then 
$$\int_{D(y,R)^c}\abs{\ip{k_y}{l_z}}p(z)d\la(z)\lesssim \int_X I(x,y)d\la(x)= \left( \int_{D(y, R/2)}I(x,y)d\la(x) +\int_{D(y,R/2)^c}I(x,y)d\la(x)\right).$$

We estimate each of the last two integrals separately. For the first integral, notice first that for every $x\in D(y,R/2)$ we have $D(y,R)^c\subseteq D(x,R/2)^c$. Therefore,

$$ \int_{D(y, R/2)}I(x,y)d\la(x)\leq \int_{D(y,R/2)}\abs{\ip{f^1_y}{g^1_x}}\int_{D(x, R/2)^c}\abs{\ip{g^2_x}{f^2_z}}p(z)d\la(z)d\la(x).$$ Denote $$C_1(r):=\sup_{x\in G}\frac{1}{p(x)}\int_{D(x,r)^c} \abs{\ip{g^2_x}{f^2_z}}p(z)d\la(z).$$ Since the pair $(\FF^2, \GG^2)$ is weakly localized, we have that $C_1(r)\to 0$ as $r\to\infty$. Therefore,  $$\int_{D(y, R/2)}I(x,y)d\la(x)\leq C_1(R/2)\int_{D(y,R/2)}\abs{\ip{f^1_y}{g^1_x}}d\la(x)\lesssim C_1(R/2)\to 0, \hspace{.2cm}\text{ as } \hspace{.2cm} R\to \infty.$$

The other integral is even easier to estimate. Indeed,
$$ \int_{D(y, R/2)^c}I(x,y)d\la(x)\leq \int_{D(y,R/2)^c}\abs{\ip{f^1_y}{g^1_x}}\int_{D(y, R)^c}\abs{\ip{g^2_x}{f^2_z}}p(z)d\la(z)d\la(x)\lesssim C_2(R/2)\to 0,$$ 
where, similarly as before, $$C_2(r):=\sup_{y\in G}\frac{1}{p(y)}\int_{D(x,r)^c} \abs{\ip{f^1_y}{g^1_x}}p(x)d\la(x).$$

\end{proof}

For a given continuous frame $\{f_x\}_{x\in X}$, as mentioned above, it was essentially proved in~\cite{Fut} that the collection of all $L^1$-localized operators with respect to $(\tilde{\FF}, \FF)$ forms an algebra. We prove that a similar result holds for $p$-weakly localized operators. 

\begin{prop}\label{algebra} Let $p: X\to (0,\infty)$ be a positive measurable function and let $\{f_x\}_{x\in X}$ be a $p$-weakly localized continuous frame in $\HH$. Let $\LL$ be the  collection of all $p$-weakly localized operators  with respect to $(\tilde{\FF}, \FF)$. Then $\LL$ forms an algebra. Moreover, if $\FF$ is a Parseval frame family, then $\LL$ is a $*$-algebra. 
\end{prop}

\begin{proof} It is easy to see that $\LL$ is closed under addition and scalar multiplication. Let $A, B\in \LL$. Using the expansion formula for frame families we obtain $$\abs{\ip{AB\tf_y}{f_z}}\leq \int_X\abs{\ip{B\tf_y}{f_x}}\abs{\ip{A\tf_x}{f_z}}d\la(x).$$ Therefore, by Lemma~\ref{useful} we obtain that $AB\in \LL$.

\end{proof}

The next proposition shows that every multiplier operator is $p$-weakly localized. 

\begin{prop} Let $\FF=\{f_x\}_{x\in X}$ and $\GG=\{g_x\}_{x\in X}$ be two Bessel families in $\HH$ such that the pair $(\FF, \GG)$ is $p$-weakly localized. Then for every $u\in L^\infty(X)$, the multiplier operator $T_u$ is $p$-weakly localized with respect to $(\FF, \GG)$. 
\end{prop}

\begin{proof} First, notice that since $\FF$ and $\GG$ are Bessel families, the operator $T_u$ is well defined for all $u\in L^{\infty}(X)$. We need to prove that the pair $(\TT, \GG)$ is weakly localized, where as before $\TT=\{Tf_x\}_{x\in X}$. This is a simple consequence of Lemma~\ref{useful}  due to 
$$\abs{\ip{T_uf_y}{g_z}}\leq \norm{u}_{\infty}\int_X\abs{\ip{f_y}{g_x}}\abs{\ip{f_x}{g_z}}d\la(x).$$

\end{proof}


\section{Norm Estimates of Weakly Localized Operators}\label{bwlo}

It is clear that each $p$-weakly localized operator $T$ must be bounded. In this section we provide norm estimates for such operators $T$ in terms of the size of $\sup_{x\in X}\norm{Tf_x}$. 

Let $\FF=\{f_x\}_{x\in X}$ be a continuous frame whose indexing space $(X,d,\la)$ satisfies (M1), (M2), and (M3). Let $T$ be a $p$-weakly localized operator with respect to the pair $(\tilde{\FF}, \FF)$. We next show that each decomposition $\DD_r=\{F_j\}$ of $X$ as in Definition~\ref{Covering} defines a sequence of operators $\{T_j\}$ that, loosely speaking, gives an approximate decomposition of the operator $T$. These operators $T_j:\HH\to \HH$ are given by

\begin{equation}\label{TJ}
T_jf:=\int_{F_j}\int_{G_j}\ip{f}{f_x}\ip{T\tf_x}{f_y}\tf_yd\la(x)d\la(y).
\end{equation}

\begin{prop} For any operator $T:\HH\to\HH$ which is $p$-weakly localized with respect to the pair $(\tilde{\FF}, \FF)$, the corresponding series $\sum_{j=1}^{\infty}T_j$ converges in the strong operator topology. 
\end{prop}

\begin{proof} Let $f\in \HH$. It is enough to show that the partial sums $S_nf=\sum_{j=1}^nT_jf$ form a Cauchy sequence. By the frame condition on $\tilde{\FF}$ we first have 

$$\norm{(S_n-S_m)f}^2\lesssim \sum_{j=m+1}^n\int_{F_j}\abs{\int_{G_j}\ip{f}{f_x}\ip{T\tf_x}{f_y}d\la(x)}^2d\la(y).$$

Using the fact that $T$ satisfies \eqref{Sch1}, the classical Schur test applied to the integral operator $Rf(y):=\int_{G_j}\abs{\ip{T\tf_x}{f_y}}f(x)d\la(x)$ implies:

$$\norm{(S_n-S_m)f}^2\lesssim \sum_{j=m+1}^n\int_{F_j}\abs{\ip{f}{f_y}}^2d\la(y).$$ The last term converges to $0$ as $m,n\to \infty$ since 

$$\sum_{j=1}^{\infty}\int_{F_j}\abs{\ip{f}{f_y}}^2d\la(y)=\int_X\abs{\ip{f}{f_y}}^2d\la(y)\lesssim \norm{f}^2.$$
\end{proof}

Define $A:=\sum_{j=1}^{\infty}T_j$, where the series is taken in the strong operator topology (the convergence is established in the previous proposition). Notice that since all of the operators $T_j$ depend on the decomposition $\DD_r$, the operator $A$ also depends on $\DD_r$. We next show that $T$ can be approximated arbitrarily well by the operator $A$ with an appropriate choice of $\DD_r$. How large this $r>0$ should be chosen depends on the localization function $\rho(\epsilon)$ introduced in~\eqref{rho}. 

\begin{prop}\label{approx} Let $\FF=\{f_x\}_{x\in X}$ be a continuous frame whose indexing space $(X,d,\la)$ satisfies (M1), (M2), and (M3), and let $T$ be a $p$-weakly localized operator with respect to the pair $(\tilde{\FF}, \FF)$. For any $\epsilon>0$ and any $r>\rho(\epsilon)$ the operator $A$ induced by the corresponding decomposition $\DD_r$ satisfies $\norm{T-A}<\epsilon$.
\end{prop}

\begin{proof}

Let $\epsilon>0$. Consider the integral operator $R_jf(y):=\int_{X}K_j(x,y)f(x)d\la(x)$ with a kernel defined by $K_j(x,y)=1_{G_j^c}(x)1_{F_j}(y)\abs{\ip{T\tf_x}{f_y}}$. By the definition of $F_j$ and $G_j$ we have that $K_j(x,y)=0$ unless $d(x,y)\geq r$. Using the fact that $T$ is $p$-weakly localized we have that $\rho(\epsilon)<\infty$. Let $r>\rho(\epsilon)$. By the classical Schur test  the integral operator $R_j$ induced by the corresponding decomposition $\DD_r$ has norm no greater than $\epsilon$ (as an operator from $L^2(X,d\la)$ into itself). 

Consider now the operator $A$ induced by this decomposition $\DD_r$.  Observe that for any $f\in\HH$ we have

\begin{eqnarray*}
Tf=\int_X\ip{Tf}{f_y}\tf_yd\la(y)&=&\sum_{j=1}^\infty \int_{F_j}\ip{Tf}{f_y}\tf_yd\la(y)\\&=&\sum_{j=1}^\infty \int_{F_j}\int_{X}\ip{f}{f_x}\ip{T\tf_x}{f_y}\tf_yd\la(x)d\la(y).
\end{eqnarray*}
 
Therefore,

$$(T-A)f=\sum_{j=1}^\infty \int_{F_j}\int_{G_j^c}\ip{f}{f_x}\ip{T\tf_x}{f_y}\tf_yd\la(x)d\la(y).$$ We then have

$$\norm{(T-A)f}^2\leq \frac{1}{C^2}\sum_{j=1}^\infty \int_{F_j}\abs{\int_{G_j^c}\ip{f}{f_x}\ip{T\tf_x}{f_y}d\la(x)}^2d\la(y),$$ where $C>0$ is the upper frame constant. Using the uniform norm estimate for the integral operator $R_j$, we obtain

$$\norm{(T-A)f}^2\leq \frac{1}{C^2}\sum_{j=1}^\infty \int_{F_j}\epsilon^2\abs{\ip{f}{f_y}}^2d\la(y)\leq \epsilon^2\norm{f}^2.$$

\end{proof}

The following results provide operator norm estimates of weakly localized operators in terms of their behavior on the continuous frame $\FF$ and its canonical dual. 

\begin{thm} \label{operatornorm} Let $\FF=\{f_x\}_{x\in X}$ be a continuous frame whose indexing space $(X,d,\la)$ satisfies (M1), (M2), and (M3), and let $T$ be a $p$-weakly localized operator with respect to the pair $(\tilde{\FF}, \FF)$. Then for any $0<\epsilon<1$ and $r>\rho(\epsilon\norm{T})$ we have the following estimate $$\norm{T}\leq\frac{\sqrt{ND_{Kr}}}{1-\epsilon} \sup_{y\in X}\left(\int_{D(y, (K+1)r)}\abs{\ip{T\tf_x}{f_y}}^2d\la(x)\right)^{1/2},$$
 where $D_r:=\sup_{x\in X} \la(D(x,r))$ and $K$ and $N$ are the constants from Definition~\ref{Covering}. 
 \end{thm}

\begin{proof} If $T=0$ there is nothing to prove. Otherwise, let $0<\epsilon<1$ and $r>\rho(\epsilon\norm{T})$. Consider the operator $A$ defined in  Proposition~\ref{approx} that corresponds to this $r$. We have $\norm{T}\leq \norm{T-A}+\norm{A}<\epsilon\norm{T}+\norm{A}$, and hence $\norm{T}\leq \frac{1}{1-\epsilon}\norm{A}$. 

We next estimate the norm of $A$. First, observe that 
\begin{eqnarray*}
\norm{Af}^2 &\leq& \frac{1}{C^2}\sum_{j=1}^\infty\int_{F_j}\abs{\int_{G_j}\ip{f}{f_x}\ip{T\tf_x}{f_y}d\la(x)}^2d\la(y)\\ &\leq& \sum_{j=1}^\infty\int_{F_j}C(y,(K+1)r)d\la(y)\int_{G_j}\abs{\ip{f}{f_x}}^2d\la(x),
\end{eqnarray*}
where $C>0$ is the upper frame constant and
$$C(y, r):=\int_{D(y,r)}\abs{\ip{T\tf_x}{f_y}}^2d\la(x).$$

Using the second property from Definition~\ref{Covering} we have $\la(F_j)\leq \sup_{x\in X} \la(D(x, Kr))=D_{Kr}$, and hence $$\int_{F_j}C(y,(K+1)r)d\la(y)\leq D_{Kr} \sup_{y\in F_j}C(y, (K+1)r)$$ for all $j$.

Therefore,

$$\norm{Af}^2\leq \frac{1}{C^2}D_{Kr}\sup_{y\in X}C(y, (K+1)r)\sum_{j=1}^{\infty}\int_{G_j}\abs{\ip{f}{f_x}}^2d\la(x).$$

The finite overlap property of $\{G_j\}$ implies that 

$$\norm{Af}^2\leq ND_{Kr}\sup_{y\in X}C(y, (K+1)r)\norm{f}^2,$$ where $N$ is the constant from Definition~\ref{Covering}. 

\end{proof}
As a simple consequence of Theorem~\ref{operatornorm} we obtain the following corollary. 
\begin{cor} Let $\FF=\{f_x\}_{x\in X}$ be a continuous frame whose indexing space $(X,d,\la)$ satisfies (M1), (M2), and (M3), and let $T$ be a $p$-weakly localized operator with respect to the pair $(\tilde{\FF}, \FF)$. Then for any $0<\epsilon<1$ and  $r>\rho(\epsilon\norm{T})$ we have the following estimate
$$\norm{T}\leq \frac{\sqrt{ND_{Kr}}}{1-\epsilon}\sup_{x\in X}\norm{T^*f_x},$$ where $N$ is the constant from Definition~\ref{Covering}. 
Moreover, if $\FF$ is a Parseval continuous frame, then 
$$\norm{T}\leq \frac{\sqrt{ND_{Kr}}}{1-\epsilon}\sup_{x\in X}\norm{Tf_x}.$$
\end{cor}


\section{Compactness of Weakly Localized Operators}\label{cwlo}

In this section we provide several criteria for compactness of weakly localized operators. Again, a crucial role will be played by the operators $T_j$ defined in~\eqref{TJ}. We will need the following property of these operators. As before, everywhere below we assume that $\FF=\{f_x\}_{x\in X}$ is a continuous frame whose indexing space $(X,d,\la)$ satisfies (M1), (M2), and (M3), and $T$ is a $p$-weakly localized operator with respect to the pair $(\tilde{\FF}, \FF)$.

\begin{lm}\label{simple} Each of the operators $T_j:\HH\to\HH$ defined in~\eqref{TJ} by $$T_jf:=\int_{F_j}\int_{G_j}\ip{f}{f_x}\ip{T\tf_x}{f_y}\tf_yd\la(x)d\la(y)$$ is compact.

\end{lm}

\begin{proof} Notice that since the metric $d$ is proper, we have that the closure $\overline{F_j}$ is compact. To any element $h\in \HH$, we associate a function $a\in C(\overline{F_j})$ defined by
$$a(y)=\int_{G_j}\ip{h}{f_x}\ip{T_j\tf_x}{f_y}d\la(x).$$
Let $\{h_n\}$ be a  sequence in $\HH$ bounded by $1$, and let $\{a_n\}$ be the corresponding sequence of functions in $C(\overline{F_j})$. It is easy to see that $|a_n(y)|\lesssim \norm{f_y}$ and $\abs{a_n(y)-a_n(z)}\lesssim \norm{f_y-f_z}$ with implied constants independent of $h_n$. These inequalities imply, by the Arzela-Ascoli criterion, that $\{a_n\}$ has a convergent subsequence $\{a_{n_k}\}$. We need to show that for the corresponding subsequence $\{h_{n_k}\}$ in $\HH$ we have that $\{T_jh_{n_k}\}$ converges. But this is clear since for every $g\in \HH$ with $\norm{g}\leq 1$ we have $$\abs{\ip{T_jh_{n_k}-T_jh_{n_l}}{g}}=\int_{F_j}\abs{a_{n_k}(y)-a_{n_l}(y)}\abs{\ip{\tf_y}{g}}d\la(y)\leq \norm{a_{n_k}-a_{n_l}}_{\infty}\sqrt{\la(F_j)}.$$

\end{proof}

Again, let $A=\sum_{j=1}^{\infty}T_j$, where the series is taken in the strong operator topology. Notice that even though all the partial sums in this series are compact, $A$ may not be compact. 

To prove our characterization of compactness we first estimate the essential norm of weakly localized operators.  

\begin{thm} Let $\FF=\{f_x\}_{x\in X}$ be a continuous frame whose indexing space $(X,d,\la)$ satisfies (M1), (M2), and (M3), and let $T$ be a $p$-weakly localized operator with respect to the pair $(\tilde{\FF}, \FF)$. Then for $r>\rho(\epsilon\norm{T}_{ess})$ we have the following estimate $$\norm{T}_{ess}\leq \frac{\sqrt{ND_{Kr}}}{1-\epsilon}  \limsup_{d(y,e)\to \infty}\left(\int_{D(y, (K+1)r)}\abs{\ip{T\tf_x}{f_y}}^2d\la(x)\right)^{1/2},$$  where $N$ is the constant from Definition~\ref{Covering}. 
\end{thm}

\begin{proof} If $\norm{T}_{ess}=0$ there is nothing to prove.  Otherwise, let $0<\epsilon<1$ and $r>\rho(\epsilon\norm{T}_{ess})$. Consider the operator $A$ from Proposition~\ref{approx} that corresponds to this $r$. For every $n\in \N$, we have $$\norm{T}_{ess}\leq \norm{T-\sum_{j=1}^nT_j}\leq \norm{T-A}+\norm{A_n}<\epsilon\norm{T}_{ess}+\norm{A_n},$$ where $A_n=\sum_{j=n+1}^\infty T_j$ with the limit in the sum taken in the strong operator topology. 

Therefore, for every $n\in\N$, we have $\norm{T}_{ess}\leq \frac{1}{1-\epsilon}\norm{A_n}$. We next estimate the norm of the tails $A_n$.

First observe that 

\begin{eqnarray*}
\norm{A_nf}^2&\leq& \frac{1}{C^2}\sum_{j=n+1}^\infty\int_{F_j}\abs{\int_{G_j}\ip{f}{f_x}\ip{T\tf_x}{f_y}d\la(x)}^2d\la(y)\\ &\leq& \frac{1}{C^2} \sum_{j=n+1}^\infty\int_{F_j}C(y,(K+1)r)d\la(y)\int_{G_j}\abs{\ip{f}{f_x}}^2d\la(x),
\end{eqnarray*}
where 
$$C(y, r)=\int_{D(y,r)}\abs{\ip{T\tf_x}{f_y}}^2d\la(x).$$

As in Theorem~\ref{operatornorm} we have $\la(F_j)\leq D_{Kr}$, and hence $$\int_{F_j}C(y,(K+1)r)d\la(y)\leq D_{Kr}\sup_{y\in F_j}C(y, (K+1)r)$$ for all $j$. Therefore,

$$\norm{A_nf}^2\leq \frac{1}{C^2}D_{Kr}\sup_{y\in \cup_{j\geq n+1} F_j}C(y, (K+1)r)\sum_{j=n+1}^{\infty}\int_{G_j}\abs{\ip{f}{f_x}}^2d\la(x).$$

The finite overlap property of $\{G_j\}$ implies that 

$$\norm{A_nf}^2\leq ND_{Kr}\sup_{y\in \cup_{j\geq n+1} F_j}C(y, (K+1)r)\norm{f}^2,$$ where the constant $N$ is the one from Definition~\ref{Covering}. By taking the infimum on both sides we obtain the desired estimate.

$$\norm{T}_{ess}\leq \frac{1}{1-\epsilon}\inf_{n}\norm{A_n}\leq \frac{\sqrt{ND_{Kr}}}{1-\epsilon}  \limsup_{d(y,e)\to \infty}\left(\int_{D(y, (K+1)r)}\abs{\ip{T\tf_x}{f_y}}^2d\la(x)\right)^{1/2}.$$

\end{proof}

\begin{cor}\label{norm compact} Let $\FF=\{f_x\}_{x\in X}$ be a continuous frame whose indexing space $(X,d,\la)$ satisfies (M1), (M2), and (M3), and let $T$ be a $p$-weakly localized operator with respect to the pair $(\tilde{\FF}, \FF)$. Then 
 $T$ is compact if and only if 
$$\lim_{d(e,x)\to\infty}\norm{T^*f_x}=0,$$ where $e$ is some/any point in $X$. 
\end{cor}

\begin{proof} The only if part is easy and it doesn't even require $T$ to be a weakly localized operator. Namely, since $f_x\to 0$ weakly as $d(x,e)\to\infty$ we have $\norm{T^*f_x} \to 0$ as $d(x,e)\to\infty$ since $T^*$ is compact.

The other direction is a trivial consequence of the estimate in the previous theorem. 
\end{proof}

\begin{cor} Let $\FF=\{f_x\}_{x\in X}$ be a  Parseval continuous frame whose indexing space $(X,d,\la)$ satisfies (M1), (M2), and (M3), and let $T$ be a $p$-weakly localized operator with respect to the pair $(\tilde{\FF}, \FF)$.  Then $T$ is compact if and only if 
$$\lim_{d(e,x)\to\infty}\norm{Tf_x}=0,$$ where $e$ is some/any point in $X$. 

\end{cor}

\begin{proof} As above, the only if part is a general statement valid for all bounded operators $T$. For the other direction, notice that since $\FF$ is a Parseval continuous frame we have that $T^*$  is weakly localized as well. Therefore, by the previous corollary, we have that $T^*$ is compact which implies that $T$ is compact.

\end{proof}


\section{Berezin Transform and Compactness}\label{Berezin}
In this section we will assume that the indexing metric measure space $(X, d, \la)$ has a group structure. Namely, we will assume that $X=G$ is a locally compact topological group whose topology is second countable (or equivalently first countable and $\sigma$-compact). Under these assumptions, by a theorem of Struble~\cite{Str}, there exists a metric $d$ on $G$ which is left-invariant ($d(zx, zy)=d(x,y)$ for all $x, y, z\in G$), proper, and generates the topology of $G$. Since all of the left-invariant Haar measures on $G$ are multiples of each other, we will pick and fix one of them and denote it by $\lambda$. Thus, we obtain a metric measure space $(G, d, \la)$ as before which is now also equipped with a group structure, such that both the metric $d$ and the measure $\la$ are left-invariant under the group action. This metric space is always complete. In addition, the invariance of $\la$ implies that the property (M3) holds. Finally, we will assume that $(G,d)$ is a geodesic metric space with a finite asymptotic dimension. Therefore, as before we have a metric measure space $(G, d, \la)$ such that the conditions (M1), (M2), and (M3) hold. 

A very natural continuous frame can be constructed in this setting starting with any irreducible square-integrable unitary representation of the group $G$. Namely, if $\pi:G\to \UU(\HH)$ is such a representation, then for any unit-norm $f\in \HH$, the family $f_x:=\pi(x)f$ forms a continuous frame. This continuous frame is Parseval if the Haar measure $\la$ is scaled appropriately. This Parseval frame is a starting point of the \emph{coorbit theory} of Feichtinger and Gr{\"o}chenig~\cite{FG1, FG2}. It is clear that in this important special case the following additional left-invariance of the inner product holds: 
\begin{equation}\label{invip}
\ip{\tf_x}{f_y}=\ip{\tf_{zx}}{f_{zy}}, \hspace{.3cm} \text{for all} \hspace{.3cm} x, y, z\in G.
\end{equation}
In our result that follows we will assume that \eqref{invip} holds. This assumption allows us to introduce the following class of ``translation'' operators. For each $y\in G$ we define $U_y: \HH\to \HH$ by $$U_yh=\int_G\ip{h}{\tf_x}f_{yx}d\la(x).$$ The frame condition for $\FF=\{f_x\}_{x\in G}$ assures that the integral converges and $\norm{U_y}\simeq 1$ with the implied constants only depending on the frame constants. The left invariance of $\la$ implies that $U_yf_z=f_{yz}$ for all $z\in G$. In addition, it is easy to see that the adjoint $U_y^*$ is given by  $$U^*_yh=\int_G\ip{h}{f_x}\tf_{y^{-1}x}d\la(x).$$ For $U_y^*$ we have similar formulas $U_y^*\tf_{x}=\tf_{y^{-1}x}$ and $\norm{U_y^*}\simeq 1$ with the implied constants independent of $y\in G$.

For a given linear operator $T:\HH\to \HH$ we define the \emph{Berezin transform} of $T$ to be the function $B(T):G\to \C$ given by 
\[ B(T)(x)=\ip{T\tf_x}{f_x}.\] It is clear that if $T$ is a bounded linear operator, i.e., $T\in \BB(\HH)$, then $B(T)$ is continuous function on $G$. Moreover, $B: \BB(\HH)\to C(G)$ is linear and bounded. The notion of Berezin transform was introduced by Berezin~\cite{Ber1, Ber2} in his famous work on quantization. This notion was later extended to many other settings and is widely used today especially in the theory of analytic function spaces (see e.g.~\cite{E2, Zhu1, Zhu2}). Almost always, when analyticity is present, the Berezin transform becomes injective. This is essentially due to the fact that any holomorphic (and anti-holomorphic) function in two variables is uniquely determined by its values on the diagonal (see e.g.~\cite{E}). In general, the Berezin transform may not be injective and the injectivity question is very subtle. This question was recently examined by Bayer and Gr{\"o}chenig~\cite{BG} for time-frequency localized operators (see Section~\ref{examples}). They obtained several sufficient conditions for injectivity without assuming any analyticity. It would be interesting to see if their result can be generalized to our more general setting.  

For our next result we will take the injectivity of the Berezin transform for granted (having in mind that injectivity holds in many important special cases). Our result provides a converse of the recent result of Bayer and Gr{\"o}chenig~\cite[Theorem 3.4]{BG} in the more difficult direction. 

\begin{thm}\label{berezin} Let $\FF=\{f_x\}_{x\in G}$ be a continuous frame in $\HH$ such that \eqref{invip} holds. Assume that the Berezin transform $B: \BB(\HH)\to C(G)$ is injective. Then a bounded linear operator $T$ which is weakly localized with respect to $(\tilde{\FF}, \FF)$ is compact if and only if $\lim_{d(x,e)\to \infty}B(T)(x)=0$, where $e\in G$ is the identity element in $G$. 

\end{thm}

\begin{proof} The if part is a direct consequence of Corollary~\ref{norm compact}. 

To prove the only if part, seeking contradiction, assume that $B(T)\in C_0(G)$ but $T$ is not compact. In this case, by Theorem~\ref{norm compact} there exists $r>0$ large enough such that $$0<1\lesssim \limsup_{d(y,e)\to \infty}\left(\int_{D(y, (K+1)r)}\abs{\ip{T\tf_x}{f_y}}^2d\la(x)\right)^{1/2}.$$ There exist a sequence $\{y_n\}\subseteq G$ with $d(y_n, e)\to \infty$ such that for all $n$  $$1\lesssim \int_{D(y_n, (K+1)r)}\abs{\ip{T\tf_x}{f_{y_n}}}^2d\la(x).$$ Changing variables we obtain
$$1\lesssim \int_{D(y_n, (K+1)r)}\abs{\ip{TU^*_{y_n^{-1}}\tf_{y_n^{-1}x}}{U_{y_n}f_{e}}}^2d\la(x)=\int_{D(e, (K+1)r)}\abs{\ip{U^*_{y_n}TU^*_{y_n^{-1}}\tf_{x}}{f_{e}}}^2d\la(x).$$
Since $\norm{U^*_{y_n}TU^*_{y_n^{-1}}}\simeq 1$ the sequence of operators $\{U^*_{y_n}TU^*_{y_n^{-1}}\}$ has a subsequence which converges in the weak operator topology. To avoid cumbersome notation we will keep denoting this subsequence by $\{U^*_{y_n}TU^*_{y_n^{-1}}\}$.  Denote the limit of this subsequence by $T_0$. Using the fact that the continuous frame $\FF$ is bounded and the dominated convergence theorem we obtain 
\begin{equation}\label{posit} \int_{D(e, (K+1)r)}\abs{\ip{T_0\tf_{x}}{f_{e}}}^2d\la(x)>0.\end{equation}

On the other hand, for every $x\in G$ we have $$B(T_0)(x)=\ip{T_0\tf_x}{f_x}=\lim_{n\to\infty} \ip{U^*_{y_n}TU^*_{y_n^{-1}}\tf_x}{f_x}=\lim_{n\to\infty} \ip{T\tf_{y_nx}}{f_{y_nx}}=0.$$ The last equality is due to the fact that $d(y_nx, e)=d(x, y_n^{-1}) \to \infty$, which follows from $d(e, y_n^{-1})=d(y_n, e)\to \infty$ by triangle inequality. 

Now, since the Berezin transform is injective by assumption we obtain that $T_0$ is a zero operator which obviously contradicts~\eqref{posit}. This finishes the proof.

\end{proof}

\begin{rem} Our norm and essential norm estimates show that all of the results above hold not just for $p$-weakly localized operators but also for operators belonging in the (operator norm) closure of the algebra of $p$-weakly localized operators.  

\end{rem}


\section{Examples and Some Applications}\label{examples}

We give several examples where our results apply. A more thorough treatment of these examples from a point of view which is more or less similar to ours can be found in~\cite{GMP1, GMP2, E1, Now}. For clarity we will present most of the examples in the simplest one-dimensional case. Everywhere below we will use the following notation for the basic unitary operators on $L^2(\R)$:
\begin{itemize}
\item[(i)] \textbf{Translation:} $T_af(x)=f(x-a),$
\item[(ii)] \textbf{Modulation:} $M_af(x)=e^{2\pi iax}f(x),$
\item[(iii)] \textbf{Dilation:} $D_af(x)=\frac{1}{\sqrt{a}}f(\frac{x}{a}), \hspace{.2cm} a>0.$
\end{itemize}

\subsection{Time-frequency (Gabor) analysis} A central role in the time-frequency analysis is played by the Weyl-Heisenberg group $H$ and its unitary representations on $L^2(\R)$. Recall that the underlying space of the Heisenberg group is $\R\times \R\times \R$ and the group law is $$(x_1,\xi_1, t_1)(x_2,\xi_2, t_2)=(x_1+x_2, \xi_1+\xi_2+\frac{x_1\xi_2-x_2\xi_1}{2}).$$ The following unitary square-integrable representation of $H$ is most relevant for the time-frequency analysis $$\pi(x,\xi,t)f=e^{2\pi i t+\pi i  x\xi}M_{\xi}T_xf.$$ Since the first term in this representation is unimodular, for each $\psi\in L^2(\R)$ with norm $1$ (usually called a window function) this representation generates two important Parseval continuous frames in $L^2(\R)$: $f_{x,\xi,t}=\pi(x,\xi,t)\psi, (x,\xi, t)\in H$, and $\psi_{x,\xi}=M_\xi T_x\psi, (x,\xi)\in \R^2$. The second one is the one which is usually used in practice since it is simpler and at the same time shares all of the advantages of the first one. The frame coefficient $\ip{f}{M_\xi T_x\psi}$ defines the so called  \emph{short-time Fourier transform} of $f\in L^2(\R)$ with respect to the window $\psi$. The most classical window function is the normalized Gaussian $\psi(x)=\frac{1}{\sqrt{\pi}}e^{-x^2/2}$ in which case the corresponding continuous frame is known as the Gabor continuous frame. 

The multiplier operators in the time-frequency analysis setting correspond to the family of so called time-frequency localization operators. Their properties have been thoroughly studied in the literature, see for example~\cite{CG1, CG2, FN, Wong} and the references therein. In the context of quantization these operators are known as anti-Wick operators (see e.g.~\cite{BC}). When $\psi$ is the normalized Gaussian the continuous frame $\FF=\{M_aT_b\psi\}_{(a,b)\in\R^2}$ is $L^1$-localized. So we can apply our Theorem~\ref{berezin} and obtain the following compactness criterion for anti-Wick operators (compare with~\cite{BC}). Note that strictly speaking we can only apply our result on the continuous frame $\FF=\{\pi(x,\xi,t)\psi\}_{(a,b,t)\in H}$. However, the transition from this frame to the more classical one is simple. 

\begin{cor} Let $\psi(x)=\frac{1}{\sqrt{\pi}}e^{-x^2/2}$ be the normalized Gaussian, $\psi_{a,b}=M_aT_b\psi$, and let $\sigma\in L^{\infty}(\R^2)$. Then the anti-Wick operator $$T_{\sigma}f=\iint_{\R^2}\sigma(a,b)\ip{f}{\psi_{a,b}}\psi_{a,b}dadb$$ is bounded on $L^2(\R)$. Furthermore, in this case $T_{\sigma}:L^2(\R)\to L^2(\R)$ is compact if and only if 
$$\ip{T_{\sigma}\psi_{c,d}}{\psi_{c,d}}=\iint_{\R^2}\sigma(a,b)\abs{\ip{\psi_{c,d}}{\psi_{a,b}}}^2dadb\to 0,$$ as $(c,d)\to \infty$. In particular, if $\sigma\in L^p(\R^2)$ for some $1\leq p<\infty$ then $T_{\sigma}:L^2(\R)\to L^2(\R)$ is compact.
\end{cor}

\begin{proof} We only need to prove the last part. If $\sigma\in L^1(\R^2)$, then the result follows from the Lebesgue dominated convergence theorem. In particular, the Berezin transform vanishes whenever $\sigma$ is compactly supported. Let  $\sigma\in L^p(\R^2)$ for some $1< p<\infty$ and let $\epsilon>0$ be arbitrary. There exists a compactly supported $\sigma_1$ such that $\norm{\sigma-\sigma_1}_p<\epsilon$. By H{\"o}lder's inequality we then have  $$\iint_{\R^2}(\sigma(a,b)-\sigma_1(a,b))\abs{\ip{\psi_{c,d}}{\psi_{a,b}}}^2dadb\lesssim \epsilon,$$ with the implied constant independent of $(c,d)\in \R^2$. This clearly implies that $\ip{T_{\sigma}\psi_{c,d}}{\psi_{c,d}}\to 0$, as $(c,d)\to\infty$.   
\end{proof}

Another important class of operators whose properties can be studied using time-frequency analysis is the class of pseudo-differential operators (see e.g.~\cite{GH, Gro2, GS}).  
It was proved by Gr{\"o}chenig and Rzeszotnik~\cite{GrR} that for a large class of window functions $\psi$ every pseudo-differential operator with a symbol $\sigma$ in the H{\"o}rmander class $S^0_{0,0}$ is $L^1$-localized relative to the pair $(\FF,\FF)$, where $\FF=\{M_aT_b\psi\}_{(a,b)\in\R^2}$. Therefore, our results about compactness apply to this class of operators as well. It is quite likely that these results can be extended to a larger class of symbols $\sigma$ if we require the corresponding operator to be only weakly localized (instead of $L^1$ localized). However, exploring this problem will require a considerable technical effort and will be hopefully elaborated elsewhere. 

\subsection{Wavelet analysis} The role of the Heisenberg group in wavelet analysis is played by the ``$ax+b$''-group $A$ (also known as the affine group). Recall that the underlying space of this group is $\R^+\times\R$ and its operation is given by $(a_1, b_1)*(a_2,b_2)=(a_1a_2, a_1b_2+b_1)$. The left-Haar measure in $A$ is $\frac{dadb}{a^2}$ and the left-invariant metric is the Riemannian metric given by the length element $ds^2=\frac{da^2+db^2}{a^2}$. 

The unitary representation of the group $A$ on $L^2(\R)$ given by $\pi(a,b)f=D_aT_bf$ generates a Parseval continuous frame $\FF=\{D_aT_b\psi\}_{(a,b)\in \R^+\times \R}$ for every $\psi\in L^2(\R)$, such that $$\int_\R\frac{|\hat{\psi}(\xi)|^2}{|\xi|}d\xi=1.$$
In this case the frame coefficient $\ip{f}{D_aT_b\psi}$ defines the so called \emph{continuous wavelet transform} of $f\in L^2(\R)$. 

The multiplier operators in wavelet analysis correspond to the family of Calderon-Toeplitz operators which are another well-studied class of operators (see for example~\cite{Ro1, Hut, Now} and the references therein). It should be noted that all of the consequences of our results for this class of operators were already proved by Nowak~\cite{Now} in the case when the wavelet function $\psi$ satisfies the following pointwise localization assumption: there exists $M>1/2$ and a constant $C$ (depending on $M$) such that \begin{equation}\label{loc}\ip{\psi_{(a_1,b_1)}}{\psi_{(a_2,b_2)}}\leq Ce^{-Md((a_1,b_1),(a_2,b_2))},\end{equation} for all $(a_1,b_1), (a_2,b_2)\in \R^+\times\R$. Our results, however, can be applied to a  more general class of wavelet functions $\psi$. Namely, as mentioned in~\cite{Now}, the Haar wavelet doesn't satisfy the almost diagonality condition~\eqref{loc}. However,  computation shows that the continuous Haar wavelet is $p$-weakly localized for $p(a,b)=b^{1/2-\delta}$, where $\delta>0$ is chosen to be small. So all our results apply to this class of operators as well.

Another important class of operators whose properties can be studied using wavelet analysis is the class of Calderon-Zygmund singular operators (see e.g.~\cite{CM, FJW}). Due to the observation of Christ~\cite[page 54]{Chr} we conjecture that our class of operators which are $p$-weakly localized relative to the continuous Haar wavelet frame and the weight $p(a,b)=b^{1/2-\delta}$ coincides with the class of weakly-bounded Calderon-Zygmund operators $T$ satisfying the conditions $T(1)=T^t(1)=0$ . This class of singular operators  plays an important role in the proof of the famous $T(1)$-theorem (see e.g.~\cite{CM}). If this conjecture turns out to be true several properties of these singular operators will follow from our results. In particular, the fact that this class of operators forms an algebra which was proved by Meyer~\cite{Mey} will be a direct consequence of our Proposition~\ref{algebra}.

\subsection{Bergman spaces} Recall that for a given bounded domain $\Om\subseteq \C^n$, the Bergman space $\Aa^2(\Om)$ is the space of holomorphic square-integrable (with respect to the Lebesgue measure) functions on $\Om$. It is well known that this space is a reproducing kernel Hilbert space. Denote by $K_z$ the reproducing (Bergman) kernel at $z$, and by $k_z$ the normalized one ($k_z=K_z/\norm{K_z}$). The normalized reproducing kernels $\{k_z\}_{z\in \Om}$ form a Parseval continuous frame in $\Aa^2(\Om)$ indexed by the metric measure space $(\Om, d, \la)$, where $d$ is the Bergman metric in $\Om$ and $\la$ is the measure $dA(z)/\norm{K_z}^2$, with $dA$ being the Lebesgue measure. 

To be able to apply our results in this context we need to make sure that all of the conditions (M1), (M2), and (M3) hold for the metric measure space $(\Om, d, \la)$. In addition we need to make sure that the Parseval continuous frame consisting of normalized reproducing kernels is $p$-weakly localized for some weight $p$. Some of these conditions, such as (M3), trivially hold, but the other ones are quite subtle and are very much dependent on the properties of the domain $\Om$. It is known that for domains which are hyperconvex (have bounded plurisubharmonic exhaustion function) the Bergman metric is complete and the normalized kernels are weakly null. Therefore, for these domains our condition (M1) holds. The finite asymptotic dimension of $(\Om, d)$ (and hence (M2)) can be established for a large class of domains $\Om$ as a consequence of Theorem~\ref{GrRoe}. For example, it was proved by Balogh and Bonk~\cite{BaBo} that every strictly pseudoconvex domain with a $C^2$-smooth boundary equipped with the Bergman metric is Gromov hyperbolic, and hence it has a finite asymptotic dimension. Also, very recently Zimmer~\cite{Zim} proved that every bounded convex domain in $\C^n$ with smooth boundary equipped with the Bergman metric is Gromov hyperbolic if and only if it is of finite type in the sense of D'Angelo. 

Finally, in the classical Bergman space (when $\Om$ is the unit ball in $\C^n$) the localization is usually proved as a consequence of the Forelli-Rudin estimates~\cite{ForRud}. These type of estimates also hold in a large class of bounded symmetric domains~\cite{FK, EZ} and in strictly pseudo convex domains~\cite{EHT}. Finding the largest class of domains where weak localization of the normalized reproducing kernels hold is still a wide open problem. 

The class of multiplier operators relative to the continuous frame of normalized reproducing kernels in the Bergman space coincides with the class of Toeplitz operators. Therefore, when specialized to this setting, our results provide norm and essential norm estimates for the class of Toeplitz operators in all Bergman spaces that satisfy the above mentioned conditions.  More precisely, we have the following general result:

\begin{cor}\label{bergman} Let $\Om$ be a bounded domain in $\C^n$ such that when equipped with the Bergman metric $d$ the metric space $(\Om, d)$ satisfies the conditions (M1), (M2), and (M3). Then an operator $T$ in the Toeplitz algebra is compact if and only if $\lim_{z\in\partial\Om}\norm{Tk_z}=0$. In particular, if $\Om$ is a bounded symmetric domain on which $p$-weak localization holds for some weight $p$, then an operator $T$ in the Toeplitz algebra is compact if and only if its Berezin transform $B(T)(z)=\ip{Tk_z}{k_z}_{\Aa^2(\Om)}$ vanishes at the boundary of $\Om$. 
\end{cor}
A very important first step towards this result was first proved by Axler and Zheng in~\cite{AZ}. Their result is for $\Om$ equal to the unit disc and applies only for operators in the algebraic part of the Toeplitz algebra, i.e., finite sums of finite products of Toeplitz operators with bounded symbols. This result was extended by Engli{\v{s}}~\cite{E1} in the context of bounded symmetric domains. Finally, Suarez~\cite{Sua} proved the general form of this result (and much more) when $\Om$ is the unit ball in $\C^n$. Several generalizations of this results of Suarez were given recently in~\cite{MSW, MW, XZ, IMW, CS}. All of these results except the one in~\cite{CS} can be obtained as a consequence of Corollary~\ref{bergman}. It would be interesting to see if the techniques from~\cite{CS} can be incorporated in our treatment to extend our result to more general domains $\Om$.

Finally, due to the following well known relationship between Toeplitz and Hankel operators $H_u^*H_u=T_{|u|^2}-T_{\bar{u}}T_u$, our results can be used to study the boundedness and compactness of Hankel operators as well (this is a well known idea). In particular, we have the following result (compare with~\cite{A, BCZ}):

\begin{cor} Let $\Om$ be a bounded domain in $\C^n$ such that when equipped with the Bergman metric $d$ the metric space $(\Om, d)$ satisfies the conditions (M1), (M2), and (M3). 
The Hankel operator $H_u$ with symbol $u\in L^{\infty}(\Om)$ is compact if and only if $\norm{H_u^*H_u k_z}_{\Aa^2(\Om)}\to 0$ as $z\to\partial\Om$. In particular, if $\Om$ is a bounded symmetric domain on which $p$-weak localization holds for some weight $p$, then a Hankel operator $H_u$ is compact if and only if $B(H_u^*H_u)(z)=\norm{H_uk_z}^2_{\Aa^2(\Om)}$ vanishes at the boundary of $\Om$. 
\end{cor}

\subsection{de~Branges space} Let $E$ be a Hermite-Biehler entire function, i.e., entire function with no zeros in $\C_+$ such that $\abs{E(\bar{z})}<\abs{E(z)}$ for all $z\in \C_+$. Each such function generates a de~Branges space $\BB_E$ consisting of all entire functions $F$ such that both $F(z)/E(z)$ and $\overline{F(\bar{z})}/E(z)$ belong in the Hardy space $\HH^2(\C_+)$. It is well known that each de~Branges space is a reproducing kernel Hilbert space when equipped with the norm $\norm{F}_{\BB_E}:=\norm{F/E}_{L^2(\R)}$. The most famous example is the Paley-Wiener space which is obtained when $E(z)=e^{-i\pi z}$. On the real line $E$ can be represented as $E(x)=\abs{E(x)}e^{-i\phi(x)}$, where $\phi(x)$ is the increasing branch of the argument of the unimodular function $E(x)/\abs{E(x)}$ on the real line. This function $\phi$ is called the phase function.

As in the Bergman space, the set of normalized reproducing kernels $\{k^E_x\}_{x\in\R}$ forms a Parseval continuous frame in $\BB_E$. Here the indexing metric measure space is the real line equipped with the metric $d(x,y)=\abs{\phi(x)-\phi(y)}$ and the measure $d\la(x)=\phi'(x)dx$. This metric space is clearly locally compact and complete. Furthermore, $\la(D(x,r))=r$ for all $x\in\R, r>0$. The remaining properties that we need to be able to apply our results are this metric space to have finite asymptotic dimension and the normalized reproducing kernels to be weakly $p$-localized for some weight $p$. These requirements are more interesting and subtle and remain to be understood. We just mention that some covering results in~\cite{MNO} seem to be closely connected to the finite asymptotic dimension property that we require.

The class of multiplier operators on de~Branges spaces again coincides with the corresponding class of Toeplitz operators. They are closely related to the class of truncated Toeplitz operators~\cite{Bar, GR}.  It is known that the normalized reproducing kernels in the Paley-Wiener space are not $p$-weakly localized no matter which weight $p$ we choose. There are however results close in spirit to ours in which additional conditions are used to derive the boundedness and compactness of the corresponding Toeplitz operators~\cite{Ro, Sm}.

\section*{Acknowledgments} We would like to thank the two anonymous referees for their comments and suggestions.

\end{document}